\documentclass[12pt]{article}
\usepackage[utf8]{inputenc}
\usepackage[margin=1in]{geometry}
\usepackage{natbib}
\usepackage{graphicx}
\usepackage{amsmath,amssymb,amsthm}
\usepackage{caption}
\usepackage{subcaption}
\usepackage{hyperref}
\usepackage[doublespacing]{setspace}  

\usepackage{xcolor}

\usepackage{lineno}

\newcommand{\mcal}[1]{\mathcal{#1}}

\newcommand{\Ind}[1]{I\!\left\{ #1 \right\}}

\newtheorem{proposition}{Proposition}[section]
\newtheorem{lemma}{Lemma}[section]

\usepackage{xcolor}

\usepackage{xr}
\newcommand{\R}[1]{\label{#1}} 

\begin{document}

\title{Optimality of testing procedures for survival data in the non-proportional hazards setting\footnote{Manuscript accepted for publication on Biometrics on May 27, 2020.}}
\date{}

\author{Andrea Arfè\footnote{Email: \url{andrea_arfe@hms.harvard.edu}. Website: \url{andreaarfe.wordpress.com}.} $^{1,2}$, Brian Alexander$^2$, Lorenzo Trippa$^{2,3}$}
\date{\begin{small}
$^1$Harvard-MIT Center for Regulatory Science, Harvard Medical School, Boston, Massachusetts.
$^2$Department of Data Sciences, Dana-Farber Cancer Institute, Boston, Massachusetts.
$^3$Department of Biostatistics, Harvard T.H. Chan School of Public Health, Boston, Massachusetts.
\end{small}}



\maketitle

\begin{abstract}
Most statistical tests for treatment effects used in randomized clinical trials with survival outcomes are based on the proportional hazards assumption, which often fails in practice. Data from early exploratory studies may provide evidence of non-proportional hazards which can guide the choice of alternative tests in the design of practice-changing confirmatory trials. We developed a test to detect treatment effects in a late-stage trial which accounts for the deviations from proportional hazards suggested by early-stage data. Conditional on early-stage data, 
among all tests which control the frequentist Type I error rate at a fixed $\alpha$ level, 
our testing procedure maximizes the Bayesian predictive  probability 
that the study will demonstrate the efficacy of the experimental treatment. Hence, the proposed test provides a useful benchmark for other tests commonly used in the presence of non-proportional hazards, for example weighted log-rank tests. We illustrate this approach in simulations based on data from a published cancer immunotherapy phase III trial. 
\end{abstract}


\section{Introduction}

Researchers often use data generated by early-phase clinical studies to specify the protocol of randomized confirmatory phase III trials. Data predictive of confirmatory trial outcomes, including early estimates of treatment effects, are used to choose the primary endpoints \citep{Gomez2014}, the sample size \citep{Lindley1997}, the target populations \citep{Lee2018}, and other aspects of the study design \citep{Brody2016}. Still, in most cases prior information is not used to specify in the protocol, as mandated by regulatory agencies, which hypothesis testing procedure will be used in the final analyses to provide evidence of treatment effects. Agencies such as the U.S. Food and Drug Administration require the control of Type I and II errors at pre-specified rates \citep{USFDA1998}. 

In Phase III trials, standard tests, such as Mantel's log-rank, are often selected even for studies where prior data suggest their assumptions will be violated \citep{Royston2013,Alexander2018}. For survival endpoints, methods related to the log-rank test are prevalent. Asymptotically, this is the most powerful test with a proportional hazards alternative \citep{Fleming2011}. However, the proportional hazards assumption is often violated in practice, contributing to false-negative findings \citep{Royston2013}, invalidating sample size calculations \citep{Barthel2006}, and affecting interim analyses \citep{Houwelingen2005}. 

Data from early-stage studies can inform about deviations from the proportional hazards assumption, suggesting the use of alternative methods  \citep{Royston2013}. Several extensions and alternatives are available to replace Mantel's test, such as weighted \citep{Fleming2011} or adaptive log-rank tests \citep{Yang2010}, and restricted mean survival tests \citep{Royston2013}. Some of these  procedures identify the most powerful test against specific alternatives, which may  be  discordant with early estimates from previous trials. Moreover, their optimality typically holds in a large-sample sense (e.g. in the local limit, for weighted log-rank tests;  \citealp{Fleming2011}). 

We develop a statistical test to detect treatment effects in late-stage trials, accounting for deviations from the proportional hazards assumption  indicated by early-phase studies (e.g. phase II trials). The proposed test does not belong to the weighted log-rank family or other common classes of tests. Starting from decision theory principles \citep{Berger2013}, we derive this test as the solution of a constrained decision problem \citep{Ventz2015}: conditional on early-stage data, the test maximizes the predicted finite-sample power among all tests which control the frequentist Type I error rate of the late-stage study at a fixed $\alpha$ level. More precisely, the test maximizes the Bayesian predictive probability that the null hypothesis will be correctly rejected at the end of the confirmatory trial. The test is therefore a useful benchmark for other procedures applicable in the presence of non-proportional hazards.

As a motivating example, we consider the analysis of a randomized trial with delayed treatment effects on survival outcomes. This  characteristic occurs when the treatment requires an induction period before it starts to provide therapeutic effects. When treatment effects are delayed, the hazard functions are not proportional and they separate across arms only later during  follow-up \citep{Fine2007}. Initially overlapping survival curves (c.f. Figure \ref{fig:sub1}) are well documented in trials of cancer immunotherapies \citep{Chen2013, Alexander2018}. They can also be observed in other settings, such as in studies of breast cancer \citep{Mehta2012} and melanoma \citep{Robert2015} chemotherapies.  

\section{Example}\label{sec:motivatingexamples}

We consider  the survival times of the 361 patients with head and neck carcinomas that participated in CheckMate 141 study \citep{Ferris2016}, a Phase III trial that randomized patients to receive nivolumab, a novel cancer immunotherapy, or standard of care (SOC) in a 2:1 ratio. We reconstructed the individual-level data of this trial from Figure 1a of \citet{Ferris2016} by means of the DigitizeIt (TM) software (version 2.2) and the data extraction method of \citet{Guyot2012}. Figure \ref{fig:sub1} shows the resulting Kaplan-Meier curves, which compare survival probabilities between the two study arms. These do not clearly separate in the initial 3-4 months of follow-up, a signal of delayed treatment effects. 

\begin{figure}[!h]
\captionsetup[subfigure]{justification=centering}
\centering
\caption{Panel a, reconstructed Kaplan-Meier curves from the CheckMate 141 trial and posterior estimates obtained from the piecewise exponential model (Section \ref{sec:pem}; dotted lines represent the position of the cut-points). Panel b, Monte Carlo estimates of the rejection probability of selected tests (Section \ref{sec:sims}). Legend: permutation, maximum-BEP test of Section \ref{sec:opttest} based on the piecewise exponential model (highlighted in red); adaptive, adaptive log-rank test of \citet{Yang2010}; mantel, classical Mantel's log-rank test; $G^{0,1}$, Fleming-Harrington weighted log-rank test; lagged, lagged-log rank that ignores the first $10\%$ of observed follow-up times \citep{Zucker1990}, RMST, test of the difference in restricted mean survival times \citep{Huang2018}.}
\begin{subfigure}[t]{0.45\linewidth}
\subcaption{}\label{fig:sub1}
\includegraphics[scale=0.4,trim={0cm 0cm 0cm 0cm}]{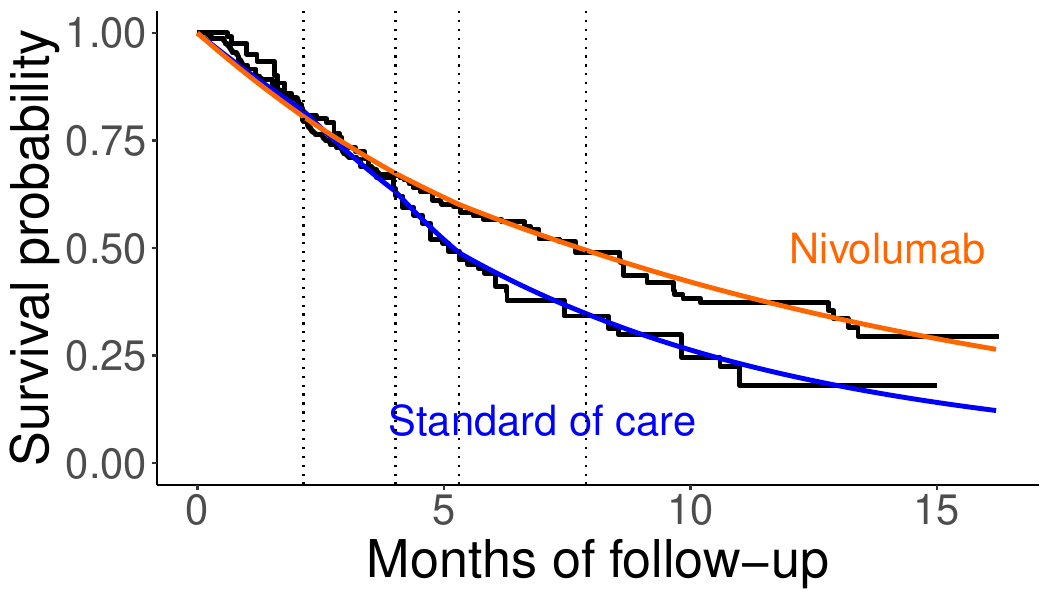}
\end{subfigure}
\hspace{1em}
\begin{subfigure}[t]{0.4\linewidth}
\centering
\subcaption{}\label{fig:sub2}
\includegraphics[scale=0.42,trim={0cm 0cm 0cm 0cm}]{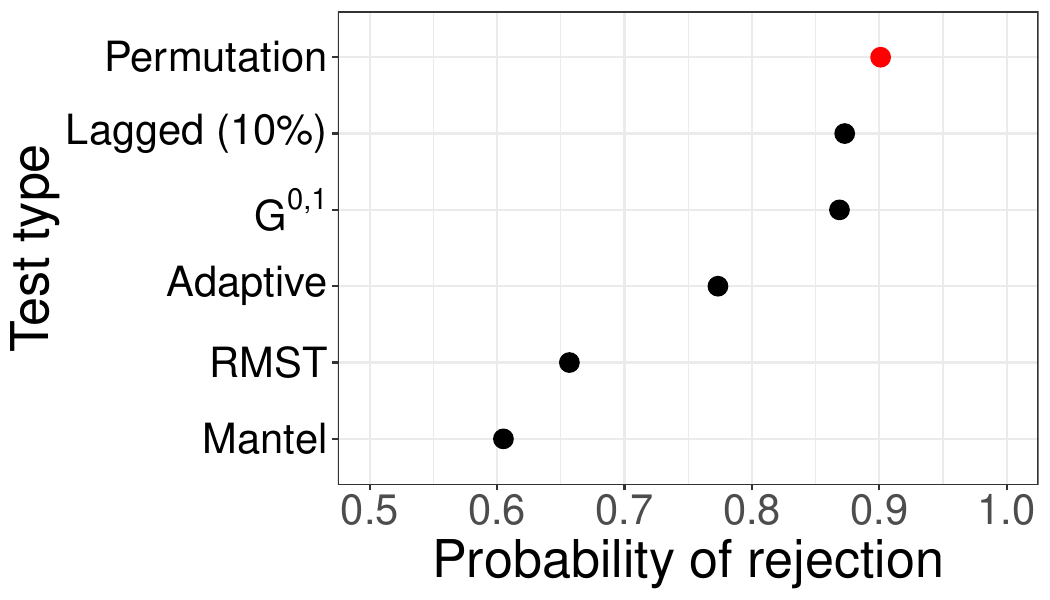}
\end{subfigure}
\label{fig:1}
\end{figure}

\section{Planning a late-stage trial}\label{sec:assumptions}

We plan a late-stage randomized trial with a survival endpoint and a sample size of $n$ patients. This will generate data $x=(t,d,a)$ to test  if the treatment has positive effects on the primary outcome. Here, $t=(t_1,\ldots,t_n)$ are the observed follow-up times, $d=(d_1,\ldots,d_n)$ are the corresponding censoring indicators ($d_i=1$ if $t_i$ is censored, while $d_i=0$ if an event was observed), and $a=(a_1,\ldots,a_n)$ are the treatment arm indicators ($a_i=0$ or $a_i=1$ if the $i$-th patient  is randomized to the control or treatment arm). Patients are assigned to arms with a fixed randomization probability. We assume that censoring times are \textit{non informative} in the sense of \citet{Heitjan1991} and independent of treatment assignment. 


For design purposes, we specify a model for the distribution that will generate the data $x$. 
This is described by a density $p_\theta(x)$ that depends on a vector of parameters $\theta\in\Theta$. Here $\theta$ may be infinite-dimensional if the model is semi- or non-parametric. Typically, $p_\theta(x)$ will have the form 
\begin{equation}\label{eqn:distr}
p_\theta(x) = \prod_{i=1}^n r^{a_i}(1-r)^{1-a_i} h_{a_i}(t_i;\theta)^{1-d_i} S_{a_i}(t_i;\theta) g_i(t_i)^{d_i}G_i(t_i)^{1-d_i},
\end{equation}
where $r\in(0,1)$ is the probability of assignment to arm $a=1$, $h_a(t;\theta)>0$ is the hazard function of arm $a=0,1$ (for example, in the exponential model, $h_a(t;\theta)=\theta_a$, $\theta=(\theta_0,\theta_1)\in \Theta=(0,+\infty)^2$), $S_a(t;\theta)=\exp\left(-\int_0^t h_a(s;\theta)ds\right)$ is the corresponding survival function, and $g_i(t)$ and $G_i(t)$ are the density and (left-continuous) survival function of the $i$-th patient's censoring time. Here, the censoring mechanism is taken as known, a common assumption when planning new experiments \citep{Chow2007}. We will later explain that this assumption is  not used in the development of the proposed testing procedure.  

We consider the non-parametric null hypothesis $H_0: P\in\mcal{P}_0$, where $P$ is the true data-generating distribution of $x$ (i.e. $P(A)$ is the probability that $x\in A$) and $\mcal{P}_0$ is the class of all distributions which are invariant with respect to permutations of the treatment arm assignments. Hence, $P\in\mcal{P}_0$ if $P(t,d,a)=P(t,d,a')$ for all $a'$ obtained by permuting the elements of $a$. 

The alternative hypothesis is defined from model (\ref{eqn:distr}) as $H_1$: $P$ has density $p_\theta(x)$ for some $\theta\in\Theta_1$, where $\Theta_1$ is a subset of $\Theta$. For example, $\Theta_1$ may include all $\theta$ such that $h_0(t;\theta)\neq h_1(t;\theta)$, such that the median of $S_1(t;\theta)$ is greater than that of $S_0(t;\theta)$, or such that the restricted mean survival time is greater in arm $a=1$ \citep{Royston2013}. 


According to the definition of the null hypothesis, 
when the experimental treatment has no effects -- regardless of whether the model $p_\theta(x)$ is correct or not -- the treatment assignments $a_1$, $\ldots$, $a_n$ do not provide information about the follow-up times $t$ and censoring indicators $d$. Hence, the distribution of the data does not change if these are arbitrarily permuted \citep{Fisher1935, Dawid1988, Good2006, Pesarin2010}. 

This definition of $H_0$ covers distributions $P(x)$ for which $(t_1,d_1,a_1),\ldots,(t_n,d_n,a_n)$ - the observations from individual patients - are not independent and identically distributed. This may happen when investigators selectively enroll patients based on interim analyses or results from other studies published during the enrollment period. Additional examples include amendments of inclusion-exclusion criteria, or improvements of adjuvant therapies \citep{Tamm2014}. In these cases, if the experimental treatment has no effects, $P(x)$   remains invariant when the indicators of the assigned arms are permuted. 

It is now necessary to choose which $\alpha$-level test $\varphi(x)$ should be used in the late-stage trial. 
A \emph{(randomized) test} of $H_0$ is a function $\varphi(x)\in[0,1]$. After the data $x$ have been collected, the 
hypothesis $H_0$ is rejected with probability $\varphi(x)$ \citep{Lehmann2006}. A test $\varphi(x)$ is \emph{non-randomized} if it can only attain the values 0 and 1. Only non-randomized tests are used in practice, but here we consider randomized tests because of their analytic advantages. The expected value $E_P[\varphi(x)]=\int_{\mcal{X}_n} \varphi(x) dP(x)$ is equal to the probability of rejecting $H_0$ with data generated from the distribution $P$. If $\alpha\in(0,1)$ and $E_P[\varphi(x)]\leq \alpha$ for all $P\in\mcal{P}_0$, then $\varphi(x)$ is said to have \emph{level} $\alpha$. 

\section{Bayesian expected power}

Different $\alpha$-level tests are usually compared with respect to their \emph{power functions} $\pi_\varphi(\theta)=\int \varphi(x)p_\theta(x)dx$. If $\varphi_1(x)$ and $\varphi_2(x)$ are two $\alpha$-level tests for $H_0$ versus the simple alternative $H_1: \theta=\theta_1$, for some fixed $\theta_1\in\Theta_1$, then $\varphi_1(x)$ is preferred to $\varphi_2(x)$ if $\pi_{\varphi_1}(\theta_1)\geq\pi_{\varphi_2}(\theta_1)$. Such comparisons are difficult for composite alternative hypotheses. In fact, uniformly most powerful $\alpha$-level tests, i.e. tests achieving the maximum power across all alternative models $\theta_1\in \Theta_1$, may not exist \citep{Lehmann2006}. 

To address this problem, some authors proposed to compare tests with respect to their \emph{average power}. Specifically, the average power of a test $\varphi(x)$ is $\int_{\Theta_1} \pi_\varphi(\theta)p(\theta)d\theta$, where $p(\theta)$ is a distribution weighting each value of $\theta\in\Theta$ based on pre-experimental information \citep{Spiegelhalter1986, OHagan2005}. With this metric, two tests are always comparable. Additionally, $\alpha$-level tests maximizing the average power always exist, although these may be randomized \citep{Chen2007}.

To allow data $x_e=(t_e,d_e,a_e)$ from an early-stage trial to inform comparisons between tests, we consider a data-dependent prior $p(\theta|x_e)$. \R{R1_support}We assume that the prior probability of the alternative hypothesis $H_1$ is positive, i.e. $P(H_1|x_e)=\int_{\Theta_1}p(\theta|x_e)d\theta>0$ (in general it may be $P(H_1|x_e)<1$ because the support of the prior may be larger than $\Theta_1$). 

Several approaches have been proposed to incorporate historical data in a prior distribution, e.g. \emph{power priors} \citep{Ibrahim2015}, \emph{meta-analytic priors} \citep{Schmid2016}, and \emph{commensurate priors} \citep{Hobbs2011}. \R{AE_prior}The effective sample size metric  of \citet{Morita2008} can guide the specification of $p(\theta|x_e)$ and help avoid overly informative choices. 

For simplicity, we define $p(\theta|x_e)$ as the posterior distribution $p(\theta|x_e)\propto L(\theta;x_e) p(\theta)$, where, letting $n_e$ be the early-stage trial sample size, 
\begin{equation}\label{eqn:likelihood}
L(\theta;x_e)=\prod_{i=1}^{n_e} h_{a_{e,i}}(t_{e,i};\theta)^{1-d_{e,i}} S_{a_{e,i}}(t_{e,i};\theta),
\end{equation}
while $p(\theta)$ is a  prior distribution on $\Theta$ whose choice depends on the specific application context. In this  definition of $p(\theta|x_e)$ we implicitly assume homogeneous treatment effects and survival distributions across the early-  and late-stage trials.


Extending the average power approach, the \emph{Bayesian expected power (BEP)} of $\varphi(x)$ is 
\begin{equation}\label{eqn:exppow}
BEP_\varphi=\int_{\Theta_1} \pi_\varphi(\theta)p(\theta|x_e)d\theta,
\end{equation}
a concept first introduced by \citet{Brown1987} and ``rediscovered'' by several authors \citep{Liu2018}. 
It is simple to observe that $BEP_\varphi=\textrm{Pr}(\varphi(x)\textrm{ rejects } H_0\textrm{ and }\theta\in\Theta_1| x_e)$, 
the probability, conditional on the early-stage data, that $\varphi(x)$ will correctly reject $H_0$ at the end of the late-stage trial. This is often called the \emph{probability of success} of the trial  \citep{Liu2018}.

From the point of view of decision theory \citep{Berger2013}, the BEP is the expected value of the utility function $u(\theta,\varphi,x)=\Ind{\theta\in\Theta_1}\varphi(x)$ (if $H_1$ holds, then the utility increases with the probability $\varphi(x)$ of rejecting $H_0$). Indeed,
\begin{equation} \label{bep}
 BEP_\varphi=\int \int u(\theta,\varphi,x)p_\theta(x)p(\theta|x_e) dx d\theta.
\end{equation} 
The problem of choosing which test to apply in the late-stage trial can thus be stated as a constrained maximization  problem \citep{Ventz2015}: among $\alpha$-level tests we optimize the BEP.


\section{Tests maximizing the expected power}\label{sec:opttest}

\R{R2_clarification} We construct an $\alpha$-level test with maximum Bayesian expected power for the null hypothesis $H_0: P\in\mcal{P}_0$, which includes all distributions invariant with respect to permutations of the treatment assignments $a_1$, $\ldots$, $a_n$. We show that the optimal test  can be obtained in the form of a \emph{permutation test}. 
Any permutation test is obtained by computing or approximating  the distribution of some real-valued \emph{test statistic} $T(x)$ across all permutations of the treatment assignments, while the values of the follow-up times $t$ and censoring indicators $d$ are kept fixed at the observed values. 


To be more formal, for each permutation $\sigma$ of $(1,\ldots,n)$, we denote with $a_\sigma=(a_{\sigma(1)}$, $\ldots$, $a_{\sigma(n)})$ the corresponding vector obtained by re-ordering the elements of $a=(a_1,\ldots,a_n)$. Moreover, if $T(x)$ is any real-valued statistics, for each $x=(t,d,a)$ we let $T^{(1)}(x)\leq \cdots \leq T^{(n!)}(x)$ be the ordered values of $T(t,d,a_\sigma)$ as $\sigma$ varies across all $n!$ permutations. 

\R{rando} The $\alpha$-level (randomized) permutation test $\varphi(x)$ of $H_0$ based on the test statistic $T(x)$ can now be defined as follows. First, let $k_\alpha=n!-\lfloor\alpha n!\rfloor$, so that, for each $x$, $T^{(k_\alpha)}(x)$ is the $(1-\alpha)$-level quantile of $T^{(j)}(x)$ for $j=1,\ldots,n!$. Second, let $M^+(x)=\sum_{j=1}^{n!}I\{T^{(j)}(x)> T^{(k_\alpha)}(x)\}$ and $M^0(x)=\sum_{j=1}^{n!}I\{T^{(j)}(x)= T^{(k_\alpha)}(x)\}$ be the number of $T^{(j)}(x)$'s greater or equal to $T^{(k_\alpha)}(x)$, respectively. Then, the permutation test $\varphi(x)$ is defined by letting $\varphi(x)=1$ when $T(x)>T^{(k_\alpha)}(x)$, $\varphi(x)=0$ when $T(x)<T^{(k_\alpha)}(x)$, and $\varphi(x)=(\alpha n! - M^+(x))/M^0(x)<1$ when $T(x)=T^{(k_\alpha)}(x)$. This satisfies $E_P[\varphi(x)]=\alpha$ for all $P\in\mcal{P}_0$ \citep[Theorem 15.2.1]{Lehmann2006}. 

In order to construct a maximum-BEP test, we define
\begin{equation}\label{eqn:marginal}
m(x) = \int_{\Theta_1} L(\theta;x)p(\theta|x_e)d\theta,
\end{equation}
which is the \emph{marginal likelihood} of the data $x$ given the early-stage data $x_e$. Without loss of generality, we assume that the density $p_\theta(x)\propto L(\theta;x)$ defined by Equation \ref{eqn:distr} is determined by a dominating measure invariant with respect to permutations of the treatment assignments. 

\begin{proposition}\label{thm:thm1} Given the early-stage data $x_e$, the $\alpha$-level permutation test with test statistic $T(x)=m(x)$ maximizes the BEP in Equation (\ref{bep}) among all $\alpha$-level tests of $H_0$.
\end{proposition}
\begin{proof}
The proof is provided in Appendix \ref{app:1}.
\end{proof}

The maximum-BEP test $\varphi(x)$ requires to compute the marginal likelihood $m(x)$ of $x$ and of all permuted versions of the same dataset. Since $m(x)$ does not depend on the censoring distribution functions  $G_i(t)$ that appear in Equation \ref{eqn:distr}, the censoring mechanism is irrelevant to identify the optimal test. Note, however, that the censoring mechanism still determines the value of the BEP of $\varphi(x)$.

\R{nonrando}Since randomized tests are not used in practice, we consider the non-randomized version of the test $\varphi(x)$ from Theorem \ref{thm:thm1}, i.e. $\varphi'(x)=\Ind{m(x)>m^{(k_\alpha)}(x)}$ (this does not depend on $M^+(x)$ or $M^0(x))$. The non-randomized test $\varphi'(x)$ rejects $H_0$ when the \emph{permutation p-value} $\textrm{ppv}(x) =\sum_\pi \Ind{m(t,d,a_\pi) \geq m(t,d,a)}/n!$ is less or equal than $\alpha$, where the sum extends over all $n!$ permutations of $a_1,\ldots,a_n$  \citep[Section 15.2.1]{Lehmann2006}. Typically, $n!$ will be too large to compute $\textrm{ppv}(x)$ exactly. Hence, we implement the test $\varphi'(x)$ by means of a Monte Carlo approximation. Accordingly, given $x$, a large random sample of permutations $\pi_1,\ldots,\pi_B$  ($B=10^3$, say) is used to compute the estimate $\widehat{\textrm{ppv}}(x)=\sum_{i=1}^B\Ind{m(t,d,a_{\pi_i}) \geq m(t,d,a)}/B.$ The hypothesis $H_0$ is then rejected if $\widehat{\textrm{ppv}}(x)\leq \alpha$ \citep[Section 1.9.3]{Pesarin2010}. 

\R{CMC}Since $\varphi'(x)\leq \varphi(x)$, the non-randomized test $\varphi'(x)$ is $\alpha$-level for $H_0$, but it may not achieve the maximum BEP. Still, $\varphi'(x)$ is a useful benchmark for other tests of $H_0$, as its BEP will be close to that of the maximum-BEP test $\varphi(x)$ for large $n$. In fact, the following Proposition shows that, under mild conditions, the difference between $BEP_\varphi$ and $BEP_{\varphi'}$ is bounded above by a known function $f(\alpha,r,n)$ such that $f(\alpha,r,n)\rightarrow 0$ as $n\rightarrow +\infty$ for all fixed levels $\alpha$ and randomization probabilities $r$. 
\begin{proposition}\label{app:prop} Suppose that for all $x=(t,d,a)$ with $q(x)>0$ it is $m(x)\neq m(t,d,a_\pi)$ for all permutations $\pi$ such that $a\neq a_\pi$. Then 
$$0\leq BEP_\varphi-BEP_{\varphi'}\leq f(\alpha,r,n)=\frac{(1-r)^n}{\alpha}\sum_{s=0}^n\left(\frac{r}{1-r}\right)^s.$$
\end{proposition}
\begin{proof}
The proof is provided in Appendix \ref{app:2}.
\end{proof}

In general, $f(\alpha,r,n)$ will be fairly small. For example, if $\alpha=0.05$ and $r=1/2$, $f(\alpha,r,n)=2^{-n}(n+1)/\alpha<10^{-3}$ for all $n\geq 15$. For $r\neq 1/2$ it is $f(\alpha,r,n)=[(1-r)^{n+1}-r^{n+1}]/\alpha(1-2r)$; for $r=2/3$ (as in Section \ref{sec:motivatingexamples}), $f(\alpha,r,n)<10^{-3}$ for all $n\geq 25$. 

\section{The piecewise exponential model}\label{sec:pem}

Because of its flexibility and tractability, to implement our maximum-BEP test we use a \textit{piecewise exponential model} \citep{Benichou1990}. In this model, the hazard function $h_a(t;\theta)$ is constant over the intervals of a fixed partition $\tau_0=0<\tau_1<\cdots<\tau_k<+\infty=\tau_{k+1}$ of the time axis. In particular, $h_a(t;\theta)=\theta_{a,j}$ if $t\in[\tau_{j-1},\tau_{j})$, with $j=1,\ldots,k+1$, $t\in\mathbb{R}_+$,  arms $a=0,1$, and $\theta=(\theta_{0,1},\ldots,\theta_{0,k+1},\theta_{1,1},\ldots,\theta_{1,k+1})\in\Theta=(0,+\infty)^{2(k+1)}$. 

The likelihood function $L(\theta;x)$ of the piecewise exponential model depends on a simple set of sufficient statistics. Given data $x=(t,d,a)$, let $s_{a,j}=\sum_{i=1}^n \max(0,\min(\tau_{j}-\tau_{j-1},t_i-\tau_{j-1})) I\{a_i=a\}$  be the total time at risk spent in the interval $[\tau_j,\tau_{j+1})$ by patients in arm $a$. 
Additionally, let $y_{a,j}=\sum_{i=1}^n (1-d_i) I\{a_i=a,\ \tau_{j-1}\leq t_i < \tau_{j}\}$ be the number of events observed during $[\tau_{j-1},\tau_{j})$ in arm $a$. Then,  
$L(\theta;x)=\prod_{a=0}^1\prod_{j=1}^{k+1}\theta_{a,j}^{y_{a,j}}\exp(-\theta_{a,j}s_{a,j}).$

For convenience, we use a conjugate prior $p(\theta)$. This is obtained by letting all $\theta_{a,j}$ be independent and distributed as a gamma random variable with shape parameter $u_{a,j}$ and rate parameter $v_{a,j}$. With this choice, the distribution $p(\theta|x_e)$  presents independent  $\theta_{a,j}$ components which are gamma  distributed  with shape parameter $u_{a,j}+ y_{e,a,j}$ and rate parameter $v_{a,j}+ s_{e,a,j}$, where the $y_{e,a,j}$ and $s_{e,a,j}$ are the sufficient statistics of $x_e$. The marginal likelihood $m(x)$ needed to implement the maximum-BEP test can thus be obtained explicitly from Equation \ref{eqn:marginal}: 
\begin{equation}\label{eqn:pemmarlik}
m(x)=\prod_{a=0}^1\prod_{j=1}^{k+1}\left(\frac{v_{a,j}+ s_{e,a,j}}{v_{a,j}+ s_{e,a,j}+ s_{a,j}}\right)^{u_{a,j}+ y_{e,a,j}+y_{a,j}}\frac{\Gamma(u_{a,j}+ y_{e,a,j}+ y_{a,j})}{\Gamma(u_{a,j}+ y_{e,a,j})},
\end{equation}
where $\Gamma(z)$ is the gamma function.

As an example, Figure \ref{fig:sub1} shows the posterior means of the survival probabilities in the nivolumab or SOC arm of CheckMate 141 obtained from the piecewise exponential model. For all $j=1,\ldots, k=4$, we conveniently defined $\tau_j$ to be the $j$-th quintile of the distribution of follow-up times in the SOC arm \R{R1_cuts}($\tau_1=2.2$, $\tau_2=4.0$, $\tau_3=5.3$, and $\tau_4=7.9$ months). In other words, the prior model is chosen by peeking at the early stage trial.
 Additionally, we specify gamma priors on the $\theta_{a,j}$ with $u_{a,j}=v_{a,j}=10^{-3}$ for all $a$ and $j$. The posterior estimates (Figure \ref{fig:sub1}) reflect the delayed separation in the Kaplan-Meier curves, as the estimated survival probabilities diverge only after  4 months of follow-up.   

\section{Application: trials with delayed treatment effects}\label{sec:app}

\subsection{Simulation study}
\label{sec:sims}

As an illustration, we use CheckMate 141 data to simulate a large number of phase II and III trials with delayed treatment effects. In these simulations, we compare different tests with respect to their probability of rejecting the hypothesis of no treatment effects at the end of the phase III trial. We consider Mantel's log-rank test and several others that account for delayed treatment effects: i) a lagged log-rank test that ignores the first $10\%$ of observed follow-up times \citep{Zucker1990}; ii) the Fleming-Harrington $G^{0,1}$ test, which gives more weight to late events \citep{Fine2007}; iii) the adaptive log-rank of \citet{Yang2010}, which weighs events according to a preliminary estimate of the hazard functions; and iv) a test of the difference in Restricted Mean Survival Times (RMSTs) across study arms \R{R1_rmst} (we follow \citealp{Huang2018}, and compute RMSTs up to the smallest of the two maximum event times observed in each study arm). We also implement the maximum-BEP test (using the Monte Carlo approach of Section \ref{sec:opttest}) based on phase II data. For all tests, we consider $\alpha=0.05$ and a two-sided alternative hypothesis.

To simulate a trial of size $n$, we first sample with replacement $n$ patients from the CheckMate 141 data. Then, depending on 
the individual treatment assignment, we generate the corresponding survival times from the Kaplan-Meier curves of Figure \ref{fig:sub1}. Assuming a maximum follow-up of 15 months, we generate patient's censoring times by sampling independently from the empirical censoring distribution \citep{Efron1981}.

Using this approach, we iterate the following steps 10,000 times: i) we simulate a phase II trial of about half the size of CheckMate 141 ($n_e=180$); \R{R1_mlik}ii) using the simulated phase II data $x_e$, we compute the sufficient statistics of the piecewise exponential model (i.e. all $s_{e,a,j}$ and $y_{e,a,j}$;
 c.f. Section \ref{sec:pem}); we fix the $\tau_j$s at the quintiles of the follow-up times in the SOC arm from $x_e$; iii) we simulate a subsequent phase III study, generating phase III data $x$ of size $n=361$; iv) we compute the marginal likelihood $m(x)$ of the phase III data $x$ (c.f. Equation \ref{eqn:pemmarlik}; we used gamma parameters $u_{a,j}=v_{a,j}=10^{-3}$); v) we apply our permutation test, as well as the others described above. The proportion of rejections across  iterations is the Monte Carlo estimate of a test's rejection probability.

Figure \ref{fig:sub2} reports the estimated rejection probabilities for each  testing procedure. The maximum-BEP permutation test based on phase II data had the highest estimated  probability of rejecting the null hypothesis (about 0.90). The $G^{0,1}$ test and the lagged log-rank test  have both  estimated rejection probabilities of about 0.87. The adaptive log-rank and RMST tests have lower rejection probabilities, 0.77 and 0.66 respectively. Mantel's log-rank test has the lowest estimated rejection probability, i.e. 0.60, a third less than the  one achieved by our test. This finding is consistent with previous studies, which highlighted how the log-rank test may suffer a severe loss of power when treatment effects are delayed \citep{Fine2007, Chen2013, Alexander2018}.


\subsection{Robustness analysis}
\label{sec:robust}

We consider four additional simulation scenarios in which the outcome distributions in phase II and III are not identical. In all scenarios, the distribution of the phase II data $x_e$ are the same as in Section \ref{sec:sims}, while the distribution of the phase III data $x$ are different. In Scenario 1, the data $x$ are generated from the predictive distribution $q(x)$ (see Equation \ref{eqn:preddist}): a value $\theta'$ is first sampled from $p(\theta|x_e)$, then $x$ is generated from the distribution $p_{\theta'}(x)$. Here we assume $r=2/3$ as in CheckMate 141 and that censoring can only occur after 15 months of follow-up. Proposition \ref{thm:thm1} indicates that, in this scenario, our permutation test has the highest expected power. Scenarios 2-4 instead represent two settings in which our test may suffer from a loss of power. In Scenario 2, $x$ is generated by a different piecewise exponential model than the one used to construct the benchmark test. The phase III delay in treatment effects is shorter than in phase II data. Specifically, $x$ is generated by a model with only one cut-point, fixed at $\tau_1=2$ months, and whose parameters are set equal to the maximum likelihood estimates obtained from CheckMate 141 data. Scenario 3 is similar, but the cut-point is fixed at $\tau_1=8$ months to represent longer phase III delays than in phase II. \R{R1.2.b.1}Finally, in Scenario 4 there are no delays in treatment effects (i.e. no cut-points) in phase III and the standard proportional hazards assumption holds. Supplementary Figure \ref{R2:KM_for_sims} shows the phase III survival curves used in each scenario.

\begin{figure}[!h]
\centering
\caption{Monte Carlo rejection probability estimates obtained in the robustness analysis of Section \ref{sec:robust}. Estimates are based on $10^4$ simulated trials. Scenario 1, best-case scenario in which our test has maximum the power among those of the same level. Scenario 2, delays shorter in phase III than phase II. Scenario 3, delays longer in phase III than phase II. Scenario 4, no delays in phase III. Supplementary Figure \ref{R2:KM_for_sims} shows the phase III survival curves assumed for each scenario. Legend: permutation, maximum-BEP test of Section \ref{sec:opttest} based on the piecewise exponential model (highlighted in red); adaptive, adaptive log-rank test of \citet{Yang2010}; mantel, classical Mantel's log-rank test; $G^{0,1}$, Fleming-Harrington weighted log-rank test; lagged, lagged-log rank that ignores the first $10\%$ of observed follow-up times \citep{Zucker1990}, RMST, test of the difference in restricted mean survival times \citep{Huang2018}.}
\includegraphics[scale=0.5,trim={0cm 0cm 0cm 0cm}]{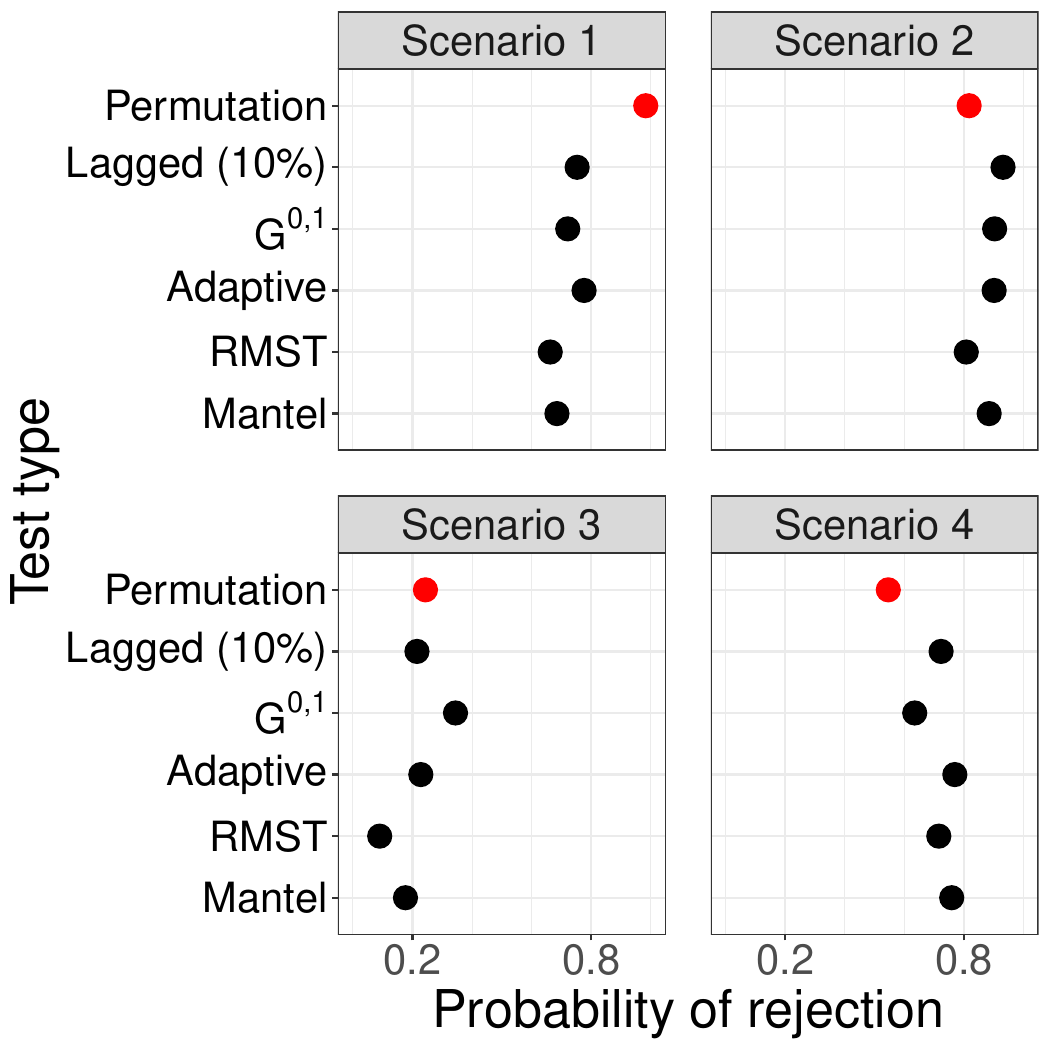}
\label{fig:new}
\end{figure}

Figure \ref{fig:new} shows the results of this robustness analysis. As expected, in Scenario 1 our permutation test has a much higher rejection probability than all other tests (0.98). Instead, its performance is sub-optimal in Scenarios 2-4. Although in Scenario 3 our permutation test may be considered comparable with the others (rejection probability equal to 0.29), in Scenario 2 it has the lowest rejection probability (0.84, compared to 0.90 for the Mantel's log-rank). \R{R1.2.b.2}The results from Scenario 4 highlight, as expected, that our test based on phase II data can have a lower power than Mantel's log-rank if the arm-specific hazards are proportional in the phase III study. All these findings support the intuition that the power of the maximum-BEP test depends on how accurately the phase III survival distributions can be predicted from prior  phase II data. If the phase II and III trial populations are markedly different, then a test specified using phase II data may perform  poorly in the phase III study. 

In the Supporting Information, we report on additional robustness analyses. \R{R1.2.c}First, we considered an additional simulation scenario where the survival curves cross (Supplementary Figure \ref{fig:cross}, panel a). In this scenario, the power of our permutation test was one of the highest among the compared tests (Supplementary Figure \ref{fig:cross}, panel b). 
 \R{R1.2.a.1}In addition, we assessed how the power of the proposed test varied with the size of the early-phase trial. To do so, we repeated the simulations of Section \ref{sec:sims} for different values of $n_e$ between 40 to 220 patients. \R{R1.2.a.2}The test's power was $\geq 80\%$ for all considered values of $n_e$, and $\geq 90\%$ for $n_e\geq 160$ (Supplementary Figure \ref{R1:phase2size}). \R{AE.3}

\section{Generalization to stratified designs}
\label{sec:markers}

Treatment effects are often expected to vary across patients' groups defined, for example, by gender or biomarkers. In such cases one can stratify patients with respect  to  covariates measured before  randomization. We focus on the primary goal  of  testing  whether the experimental  treatment has no effects across all strata or if it is effective at least in some of the strata (alternative hypothesis), for example in one or multiple  subgroups defined by a relevant biomarker \citep{Freidlin2010}. 

Our approach can be easily generalized to this setting. For simplicity, we suppose each patient $i=1,\ldots,n$ is  categorized  by  a binary covariate $z_i=0,1$, presence ($z_i=1$) or absence ($z_i=0$) of some marker. Data $x$ become $x=(t,d,a,z)$, where $z=(z_1,\ldots,z_n)\in\{0,1\}^n$. We assume that censoring is non-informative and independent of treatment assignments conditionally on $z_1, \ldots, z_n$ \citep{Heitjan1991}. 

To illustrate, we specify a piecewise-exponential model $h_a(t;\theta,z)$ for the hazard function in arm $a=0,1$ for patients with marker level $z=0,1$ \citep{Freidlin2010}: $h_a(t;\theta,z)=\theta_{a,z,j}$ for all $t\in [\tau_{j-1},\tau_j)$. The prior remains nearly identical to the  previous  sections. In particular, the marginal likelihood $m(t,d,a,z)$, similar to Equation \ref{eqn:marginal}, has a closed form expression. 

We specify the null hypothesis $H_0': P\in\mcal{P}'_0$, where $\mcal{P}'_0$ is the class of all distributions which are invariant with respect to permutations of the treatment assignment $a$ within strata of $z$. More precisely, $P\in\mcal{P}'_0$ if and only if $p$, the density of $P$, satisfies $p(t,d,a,z)=p(t,d,a_\sigma,z)$  for all permutations $\sigma$ of $(1,\ldots,n)$ such that $z_{\sigma(i)}=z_i$ for all $i=1,\ldots,n$.

With a simple modification, Proposition \ref{thm:thm1} still holds with this new definition of the null hypothesis. In Section \ref{sec:opttest}, the maximum-BEP permutation test computes  the distribution of $m(x)$ under $H_0$ by considering all permutations of the treatment arm indicators $a_1,\ldots, a_n$. 
In the stratified case, only permutations $\sigma$ of $(1,\ldots,n)$ such that $z_{\sigma(i)}=z_i$ for all $i=1,\ldots,n$ are considered. If $\Sigma(z)$ is the set of all such $\sigma$, then the permutation p-value associated to the maximum-BEP test is given by $\textrm{ppv}(x)=\sum_{\sigma\in \Sigma(z)}I\{m(t,d,a_\sigma,z)\geq m(t,d,a,z)\}/|\Sigma(z)|$, where $|\Sigma(z)|=(\sum_{i=1}^n z_i)!(n-\sum_{i=1}^n z_i)!$. 

\subsection{Simulation study of biomarker-stratified designs}\label{sec:moresims}

Similarly as in Section \ref{sec:sims}, we use CheckMate 141 data to simulate phase II and III trials where 50\% of patients express ($z_i=1$) a marker predictive of treatment effects \citep{Patel2015}. In the simulations, we compare different tests for the final analysis of the phase III trial, each accounting for patients' maker values. Specifically, we consider four (5\%-level) tests of $H_0'$:
i) a test based on a stratified Cox proportional hazards model (a common approach in this setting; \citealp{Mehrotra2012}); 
ii) the procedure obtained by first performing separate log-rank tests within marker strata and then applying the Bonferroni correction (another common approach; \citealp{Freidlin2014});
 iii) the test obtained by estimating separate log-normal accelerated failure time (AFT) models \citep[Chapter 2]{Kalbfleisch2002} within marker strata,
  testing the statistical significance of their regression coefficients, and then applying the Bonferroni correction; and, lastly, iv) our maximum-BEP test, tailored to the simulated phase II data.

In detail, we simulate 10,000 phase II ($n_e=180$) and III ($n=361$) trials from CheckMate 141 data for each of 3 scenarios. In Scenario 1, the survival distributions in phase II and III trials are the same. Here, the survival time of a patient in arm $a=0$ or with marker $z=0$ (respectively, in arm $a=1$ with marker $z=1$) is generated from the SOC (nivolumab) Kaplan-Meier curve in Figure \ref{fig:sub1}. Censoring times are generated as in Section \ref{sec:sims}.

In Scenarios 2 and 3, phase II data are generated as before, but the outcome distribution in phase III is different. In Scenario 2, phase III data are generated assuming proportional hazards within levels of $z$. Specifically, the survival time of a patient in arm $a=0$ or with marker $z=0$ (respectively, in arm $a=1$ with $z=1$) is simulated from an exponential distribution, fixed at the corresponding estimate from the SOC (respectively, nivolumab) arm of CheckMate 141. Scenario 3 is similar, but we do not assume proportional hazards in the strata of $z$. Rather, phase III data are generated as in Scenario 2, but using stratum-specific log-normal distributions \citep{Kalbfleisch2002}.

Figure \ref{fig:2} shows the results of the simulations. For Scenario 1, the estimated rejection probabilities are 0.18 for the stratified Cox model, 0.26 for the Bonferroni-based log-rank test, 0.08 for the Bonferroni-based AFT test, and 0.49 for our test tailored to early-stage data. As in Section \ref{sec:robust}, in Scenario 2 the test based on a stratified Cox proportional hazards model and the Bonferroni-based log-rank test outperforms our tailored test, as their estimated rejection probabilities are 0.21, 0.29, and 0.17, respectively. Instead, in Scenario 3, the proposed test has an estimated rejection probability (0.32) higher than that of the other tests, although the power advantage is reduced compared to Scenario 1. \R{R1M7}We obtained qualitatively the same results when we repeated the simulations assuming a phase III sample size ($n=1,000$) sufficient to achieve 90\% power with our test (Supplementary Figure \ref{fig:suppl_marker}). These results confirm that, in late-stage trials, a benefit can be attained when prior data are used to optimize hypothesis testing, but this benefit can be reduced when the early- and late-stage survival distributions are discordant. 

\begin{figure}[!h]
\centering
\caption{Monte Carlo rejection probability estimates obtained for the 3 scenarios of Section \ref{sec:moresims}, using selected tests that account for a binary predictive marker. All estimates are based on $10^4$ simulated trials. Scenario 1, homogeneous survival distributions between phase II and phase III. Scenario 2, proportional stratum-specific hazards in phase III but not in phase II. Scenario 3, log-normal stratum-specific survival distributions. Legend: permutation, maximum-BEP test of Section \ref{sec:markers} based on the piecewise exponential model (highlighted in red); stratum-specific log-rank, Bonferroni correction of two stratum-specific tests based on log-normal accelerated failure time models \citep{Freidlin2014}; stratum-specific AFT, Bonferroni correction of two stratum-specific tests based on log-normal accelerated failure time models \citep{Kalbfleisch2002}; stratified Cox, test based on a stratified Cox proportional hazards model \citep{Mehrotra2012}. }
\includegraphics[scale=0.5,trim={0cm 0cm 0cm 0cm}]{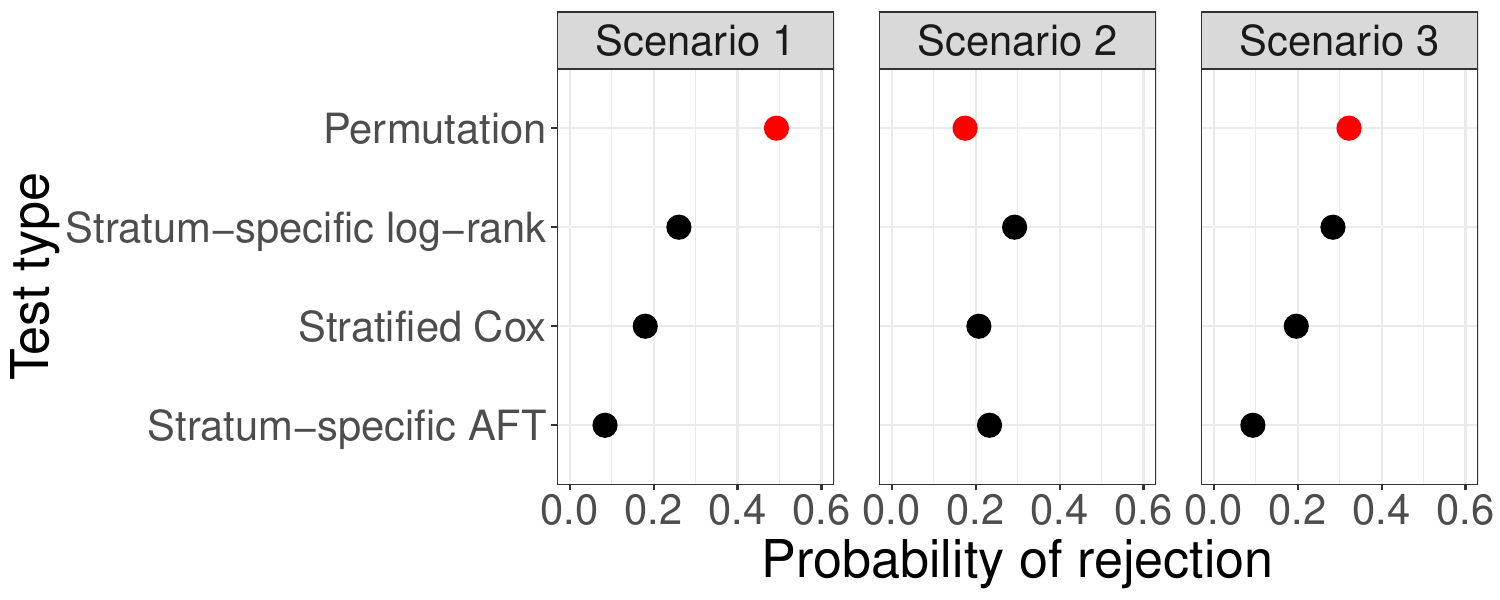}
\label{fig:2}
\end{figure}

\section{Discussion}

Data from previous studies should be routinely used to design late-stage clinical trials. This is especially relevant when standard  assumptions, such as the proportional hazards assumption, might not hold. \R{R1.2.c.disc}Our approach allows to specify a test for final analyses that achieves two goals: i) it accounts for the specific type of deviations from proportional hazards suggested by prior data - delayed treatment effects,  crossing survival curves, or others \citep{Uno2014}; ii) it satisfies the requirements of regulatory agencies \citep{Ventz2015}. The test maximizes a decision-theoretic criteria, it leverages prior data, and it is of $\alpha$-level for an interpretable null hypothesis.

\R{ext_data}Our contribution is distinct from the growing body of literature on trial designs that incorporate external data in the study analyses \citep{Psioda2018, Ventz2019}. In contrast, we leverage early-phase data only during the design stage of the late-phase trial. We do not directly use any external data in the final analysis of the late-phase trial. As a consequence, our testing procedure controls the phase III false-positive error rate at the chosen $\alpha$ level, even when early-phase data may overestimate the treatment effect. 


\R{R1_3} The use of appropriate metrics of treatment effects is  important to interpret trial results. Rejecting the null hypothesis of equality of survival curves is often not sufficient to understand which treatment is clinically preferable, especially when survival curves cross. To do so, study arms must be compared using interpretable measures of the treatment effects  \citep{Saad2018}. \R{R1_M5}
For example, in the context of Section 7 it could be relevant to consider the one-sided alternative hypothesis that the 12-months restricted mean survival time in the nivolumab arm is greater than in the standard of care arm \citep{Royston2013}. 
As discussed in Section 3, our approach could be used to construct testing procedures that maximize the expected power to detect improvements in restricted mean survival time. 

\R{R1_sample_size}To plan a phase III trial using our approach, different designs can be compared in simulation studies consistent with early-phase data \citet{Liu2018}. For example, the phase III sample size can be chosen as follows. First, multiple hypothetical phase III trials, each based on a different sample size, are simulated on the basis of early-phase information (e.g. as done in Section \ref{sec:app}). Then, the probabilities of detecting treatment effects associated to each sample size are estimated using the simulations. The smallest sample size whose estimated power achieves adequate levels (e.g. $>80\%$) is chosen for the design of the future phase III trial. 

\R{R1_1_1}To implement our approach, we used a piece-wise exponential model with fixed cut-points chosen on the basis of prior data \citep{Benichou1990}. In general, however, it may be necessary to choose between models with a different number of cut-points - e.g. by comparing their fit to early-phase data using Akaike's information criterion or similar metrics \citep[Chapter 7]{Gelman2013}. \R{R1_1_2}Models that regularize and shrink the difference between $\theta$ components between adjacent time points could be advantageous, especially with small early-phase datasets. This is because the regularization could reduce the impact of the choice of cut-points (\citealp[Chapter 3]{Ibrahim2005}; \citealp{Murray2016}). Additionally, investigators can add a prior distribution on the position and the number of cut-points \citep{Demarqui2008}. Other parametric and or non-parametric models different than the piece-wise exponential could also be considered in our approach \citep{Ibrahim2005}. Regardless of the chosen model, to implement our test it is necessary to compute the marginal likelihood of late-stage data. This may be complicated for some non-conjugate models, but several methods have been developed to approximate it \citep{Pajor2017}.

Although we derived our test assuming a single early-stage dataset, the use of multiple prior data sources may provide better outcome predictions for late-stage trials. Our approach can incorporate multiple datasets using power priors \citep{Ibrahim2000} or hierarchical models \citep{Spiegelhalter2004}. 

Our simulations, based on data from the CheckMate 141 trial, confirm  that  weighted log-rank tests can outperform other tests in the presence of delayed treatment effects. However, these tests depend on a set of tuning parameters, such as the duration of the lag time for lagged log-rank tests or the $\rho$ and $\delta$ coefficients of the $G^{\rho,\delta}$ family, which may be difficult to tune. Instead, our approach directly translates early-stage data into a test procedure for the late-stage trial. Robustness analyses highlight how the performance of our approach depends on the consistency of outcome data and the similarity of enrolled populations between phase II and III trials. Ensuring the transportability of results to subsequent trials remains a major concern in the design of exploratory clinical trials \citep{Wang2006}. 

\section*{Acknowledgements}
We thank G. Fell and T. Chen (Dana-Faber Cancer Institute) for helping with data collection. We also thank Massimiliano Russo (Harvard Medical School) for useful suggestions and Sarah Craver for help in manuscript editing.

\section*{Data Availability Statement}
The data that supports the findings of this study are available in the supplementary material of this article.

\bibliography{bibliography}
\bibliographystyle{biom}

\appendix
\section*{Appendix: technical proofs}
\renewcommand{\thesection}{A} 

\subsection{Proof of Proposition \ref{thm:thm1}}\label{app:1}

We will use two technical lemmas. The first shows that for any probability law  $Q\not\in\mcal{P}_0$ (c.f. Section \ref{sec:assumptions}) it is possible to construct a permutation test which is $\alpha$-level for $H_0:P\in\mcal{P}_0$  and has maximum power against the simple alternative $H_1:P=Q$. We assume that the densities $p_\theta(x)$ are taken with respect to the same dominating measure $\mu$, and that this is invariant with respect to permutations of the treatment assignments.

\begin{lemma}\label{lem:lemma1}
Let $\varphi(x)$ be the $\alpha$-level permutation test of $H_0:P\in\mcal{P}_0$ based on $T(x)=q(x)$, where $q(x)$ is the density of $Q\not\in\mcal{P}_0$ with respect to $\mu$. If $\varphi'(x)$ is another $\alpha$-level test of $H_0$, then $E_Q[\varphi'(x)]\leq E_Q[\varphi(x)]$, i.e. $\varphi(x)$ has higher power under the alternative $H_1:P=Q$.
\end{lemma}
\begin{proof}
Let $\mcal{P}_\mu\subseteq\mcal{P}_0$ be the set of all distributions dominated by $\mu$ that are invariant with respect to permutations of treatment assignment (a non-empty set, since it includes $q'(t,d,a)=\sum_{\sigma}q(t,d,a_\sigma)/n!$). By Theorem 2 of \citet{Lehmann1949},  $E_Q[\varphi'(x)]\leq E_Q[\varphi(x)]$ for every test $\varphi'(x)$ such that $E_P[\varphi'(x)]\leq \alpha$ for all $P\in \mcal{P}_\mu$. Now, if $\varphi'(x)$ is an $\alpha$-level test of $H_0$, then $E_P[\varphi'(x)]\leq \alpha$ for all $P\in\mcal{P}_\mu$ and   therefore $E_Q[\varphi'(x)]\leq E_Q[\varphi(x)]$.
\end{proof}

The second lemma shows that to construct an $\alpha$-level maximum-BEP test of $H_0$ it is sufficient to construct an $\alpha$-level test of $H_0$ that has maximum power against a simple alternative $H_1:P=Q$ defined by a specific $Q\not\in\mcal{P}_0$. More precisely, here $Q$ is the the predictive distribution of $x$ conditional on $\theta\in\Theta_1$, which is defined by
\begin{equation}\label{eqn:preddist}
q(x) = \int_{\Theta_1} p_\theta(x)\frac{p(\theta|x_e)}{P(H_1|x_e)}d\theta,
\end{equation}

\begin{lemma}\label{lem:lemma2}
A test $\varphi(x)$ of $H_0$ maximizes the BEP (c.f. Equation \ref{eqn:exppow}) among all $\alpha$-level tests if and only if it maximizes the power $E_Q[\varphi(x)]$ among all $\alpha$-level tests.
\end{lemma}
\begin{proof}
By Fubini's theorem, the BEP of a test $\varphi(x)$ can be written as
$BEP_\varphi = E_Q[\varphi(x)]\cdot P(H_1|x_e)$. \R{test_norm}Since $P(H_1|x_e)$ is positive and does not depend on $\varphi(x)$,  $BEP_\varphi$ is maximal if and only if $E_Q[\varphi(x)]$ is also maximal. \qedhere
\end{proof}

We are finally ready to prove  Proposition \ref{thm:thm1}.

\begin{proof}[Proof of Proposition \ref{thm:thm1}]
By Lemma \ref{lem:lemma1}, the $\alpha$-level permutation test $\varphi'(x)$ based on $T'(x)=q(x)$ maximizes the power $E_Q[\varphi(x)]$ among all $\alpha$-level tests of $H_0$. By Lemma \ref{lem:lemma2}, $\varphi'(x)$ thus has maximum BEP among all $\alpha$-level tests of $H_0$. It now suffices to show that $\varphi'(x)=\varphi(x)$ for all $x$ such that $q(x)>0$, where $\varphi(x)$ is the $\alpha$-level permutation test based on $T(x)=m(x)$. To do so, note that, by Equation \ref{eqn:distr}, if $q(x)>0$, then $m(x)>0$ as well, and the ratio $q(x)/m(x)$ is invariant with respect to permutations of the treatment arm assignments. The thesis now follows since $q(t,d,a_\sigma)\propto m(t,d,a_\sigma)$ for all permutations $\sigma$. \qedhere

\end{proof}

\subsection{Proof of Proposition \ref{app:prop}}\label{app:2}

Denote with $\Pi(x)$ the set of all $\binom{n}{\sum_{i=1}^n a_i}$ distinct datasets obtainable from $x=(t,d,a)$ by permuting the elements of $a$. Let $Q$ be defined by the density $q(x)$ in Equation \ref{eqn:preddist}. By assumption, if $q(x)>0$ then $m(x_1)\neq m(x_2)$ for all $x_1,x_2\in \Pi(x)$, $x_1\neq x_2$. By Proposition \ref{lem:lemma2}, $0\leq BEP_\varphi-BEP_{\varphi'} = E_Q[\varphi(x)-\varphi'(x)] \leq Q(E)$, where $E$ is the set of all $x$ such that $m(x)=m^{(k_\alpha)}(x)$. From Section 5.9 of \citet{Lehmann2006},
\begin{equation}\label{eqn:qdist}
Q(E)=\int \frac{\sum_\sigma \Ind{m(t,d,a_\sigma)=m^{(k_\alpha)}(x)}q(t,d,a_\sigma)}{\sum_\sigma q(t,d,a_\sigma)} dQ(x),
\end{equation}
where both sums extend over all $n!$ permutations $\sigma$ of $(1,\ldots,n)$. If $q(x)>0$, then
\begin{align*}
&\frac{\sum_\sigma \Ind{m(t,d,a_\sigma)=m^{(k_\alpha)}(x)}q(t,d,a_\sigma)}{\sum_\sigma q(t,d,a_\sigma)} = \frac{\sum_\sigma \Ind{m(t,d,a_\sigma)=m^{(k_\alpha)}(x)}m(t,d,a_\sigma)}{\sum_\sigma m(t,d,a_\sigma)} \\
&\quad\quad \leq \frac{\sum_\sigma \Ind{m(t,d,a_\sigma)=m^{(k_\alpha)}(x)}m(t,d,a_\sigma)}{\sum_\sigma \Ind{m(t,d,a_\sigma)\geq m^{(k_\alpha)}(x)} m(t,d,a_\sigma)} 
\leq \frac{\sum_\sigma \Ind{m(t,d,a_\sigma)=m^{(k_\alpha)}(x)}m^{(k_\alpha)}(x)}{\sum_\sigma \Ind{m(t,d,a_\sigma)\geq m^{(k_\alpha)}(x)} m^{(k_\alpha)}(x)} \\
&\quad\quad = \frac{\sum_\sigma \Ind{m(t,d,a_\sigma)=m^{(k_\alpha)}(x)}}{\sum_\sigma \Ind{m(t,d,a_\sigma)\geq m^{(k_\alpha)}(x)}} = \frac{\#\{j: m^{(j)}(x)=m^{(k_\alpha)}(x)\}}{\#\{j: m^{(j)}(x)\geq m^{(k_\alpha)}(x)\}},
\end{align*}
where the first equality follow because the ratio $q(x)/m(x)$ is invariant with respect to permutations $\sigma$ of $a$. Now, by the definitions of $m^{(k_\alpha)}(x)$ and $k_\alpha$, the denumerator of the last fraction is greater or equal than $ \alpha n!$. Instead, the numerator is equal to $n!/\binom{n}{\sum_{i=1}^n a_i}$, as i) $m^{(k_\alpha)}(x)=m(x_\alpha)$ for some $x_\alpha\in\Pi(x)$, ii) for each $x'=(t,d,a')\in\Pi(x)$ there are exactly $n!/\binom{n}{\sum_{i=1}^n a_i}$ permutations $\sigma$ such that $x'=(t,d,a_\sigma)$, and iii) $m(x)$ assumes distinct values on distinct points of $\Pi(x)$, by assumption. Thus, by Equation \ref{eqn:qdist},  
$$Q(E) \leq E_Q\left[\frac{1}{\binom{n}{\sum_{i=1}^n a_i}\alpha}\right] 
= \sum_{s=0}^n \frac{1}{\binom{n}{s}\alpha}\binom{n}{s}r^s(1-r)^{n-s} 
= f(\alpha,r,n).$$

\newpage

\section*{Supporting Information}

\begin{figure}[!htb]
\centering
\caption{Description of the data-generating distributions considered in the four scenarios of Section 7.2. In all scenarios, phase II data was generated from the Kaplan-Meier curves of Figure 1a of the main paper (c.f. Section 7.1). Phase III data was instead generated from the pictured survival curves. Scenario 1: predictive distributions obtained from the piece-wise exponential model; plotted curves are averages over the 10,000 simulated phase II trials (c.f. Section 7). Scenarios 2 and 3: fixed piece-wise exponential survival curves; the survival probabilities can change across arms only after 2 or 8 months of follow-up, respectively; event rates are fixed at the corresponding estimates obtained from Checkmate 141 data (c.f. Section 2). Scenario 4: similar to scenarios 2 and 3, but survival curves separate immediately and hazards are proportional across study arms. The orange (blue) curves are the assumed survival distributions in the nivolumab (standard of care) arms.  \label{R2:KM_for_sims}}
\includegraphics[scale=0.6]{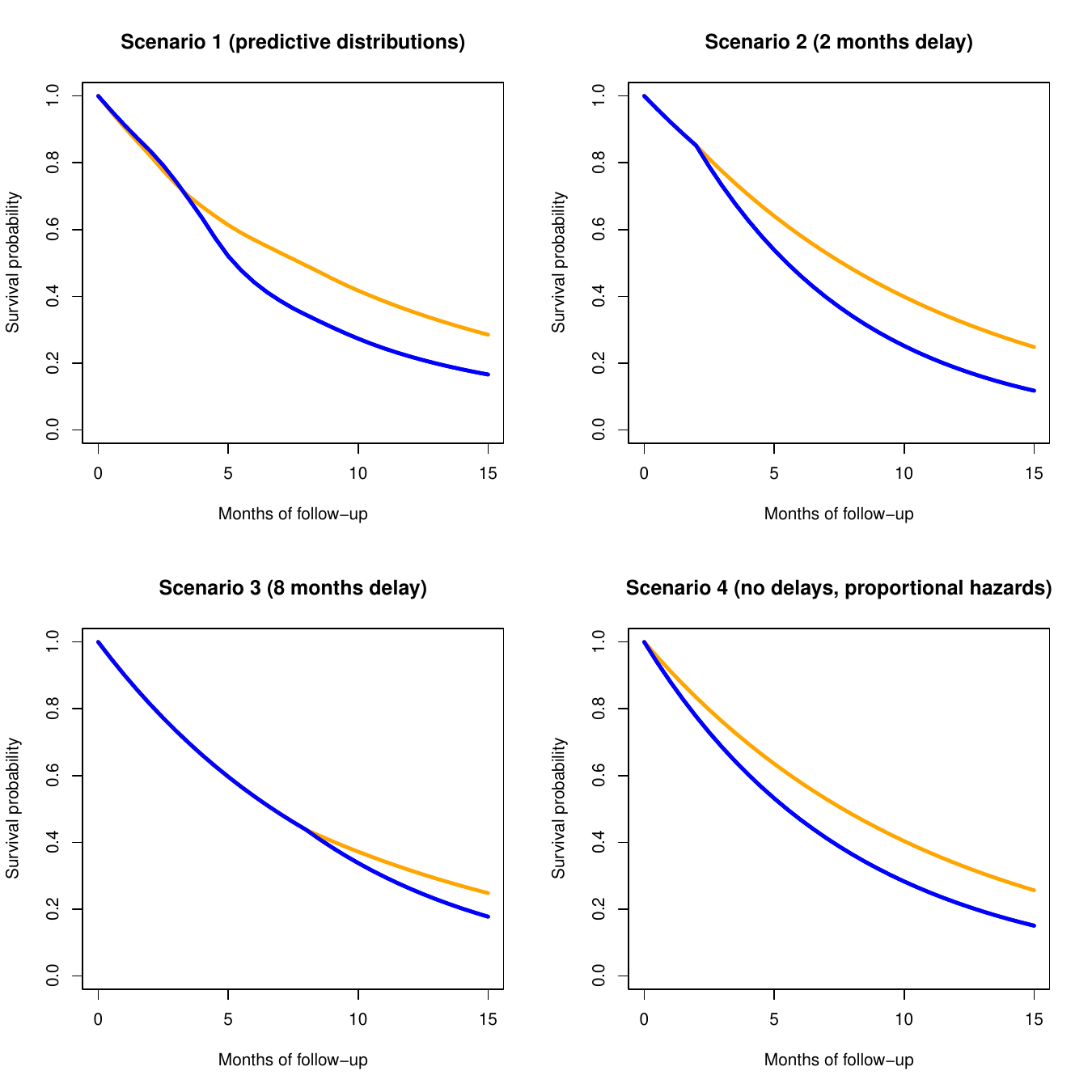}
\end{figure}

\newpage

\begin{figure}[!htbp]
\captionsetup[subfigure]{justification=centering}
\centering
\caption{Panel a, Phase III data-generating survival curves used in the additional simulations of Section 7.2. The orange (blue) curve is the assumed survival distribution in the nivolumab (standard of care) arm. Panel b, Monte Carlo estimates of the rejection probability of selected tests. The structure of the simulations was identical as in Section 7.2. Legend: permutation, maximum-BEP test based on the piecewise exponential model (highlighted in red); adaptive, adaptive log-rank test of \citet{Yang2010}; mantel, classical Mantel's log-rank test; $G^{0,1}$, Fleming-Harrington weighted log-rank test; lagged, lagged-log rank that ignores the first $10\%$ of observed follow-up times \citep{Zucker1990}, RMST, test of the difference in restricted mean survival times \citep{Huang2018}.}
\begin{subfigure}[t]{0.45\linewidth}
\subcaption{}
\includegraphics[scale=0.45,trim={0cm 0cm 0cm 1cm}]{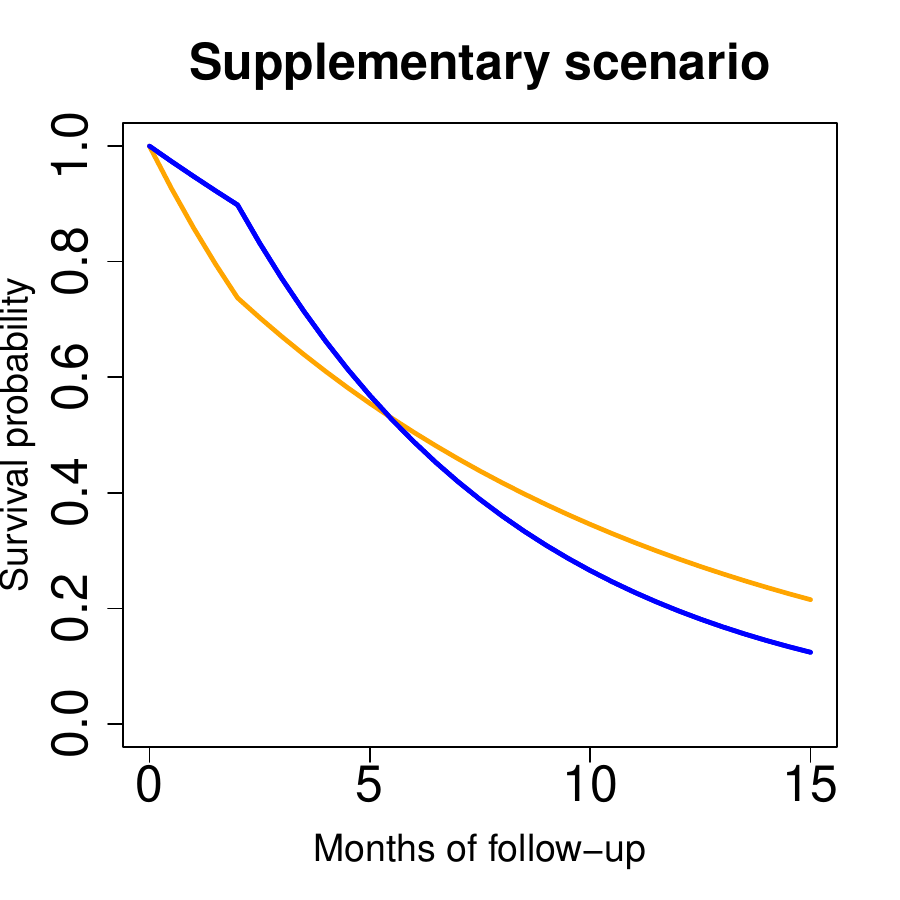}
\end{subfigure}
\begin{subfigure}[t]{0.4\linewidth}
\centering
\subcaption{}
\includegraphics[scale=0.425,trim={0cm 0cm 0cm 0cm}]{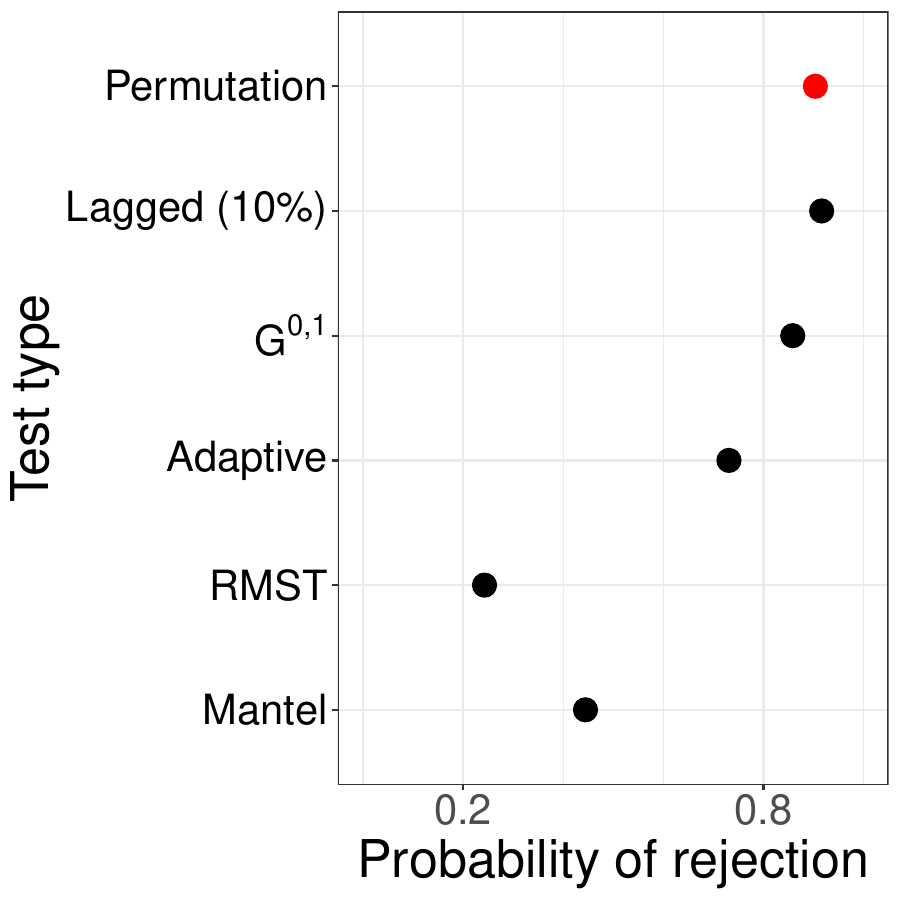}
\end{subfigure}
\label{fig:cross}
\end{figure}

\newpage

\begin{figure}[!htb]
\centering
\caption{Estimates (black dots) of the maximum-BEP permutation test's power obtained by varying early-phase samples sizes ($n_e=40,60,80,\ldots,180,200,220$) in the simulation study of Section 7.1. 
\label{R1:phase2size}}
\includegraphics[scale=0.9]{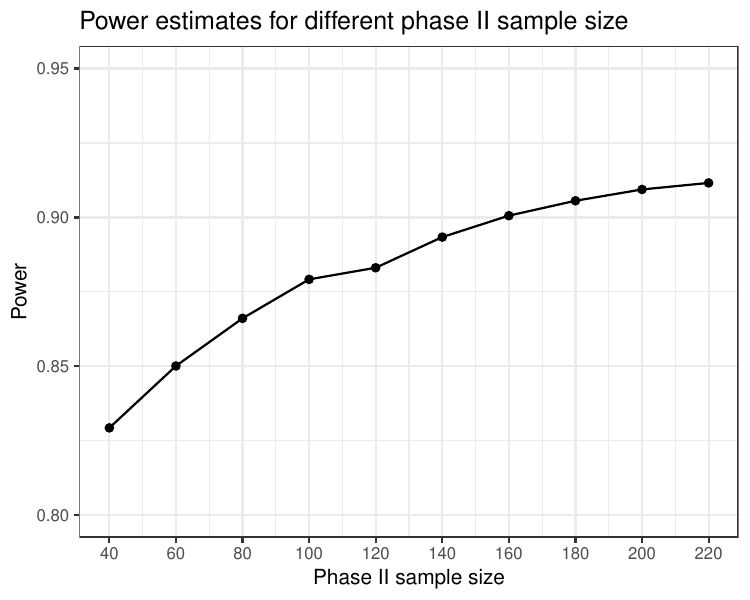}
\end{figure}

\newpage

\begin{figure}[htb]
\centering
\caption{Additional Monte Carlo rejection probability estimates for different tests that account for a binary predictive marker. Simulations were based on the same 3 scenarios of Section 8.1, but using a larger phase III sample size  ($n=1,000$). Results are qualitatively the same as to those reported in Figure 3. Legend: permutation, maximum-BEP test of Section 5 based on the piecewise exponential model (highlighted in red); stratum-specific log-rank, Bonferroni combination of two stratum-specific log-rank tests \citep{Freidlin2014}; stratum-specific AFT, Bonferroni combination or two stratum specific tests based on a log-normal accelerated failure time models \citep{Kalbfleisch2002}; stratified Cox, test based on a stratified Cox proportional hazards model \citep{Mehrotra2012}. \label{fig:suppl_marker}
}
\includegraphics[scale=0.6,trim={0cm 0cm 0cm 0cm}]{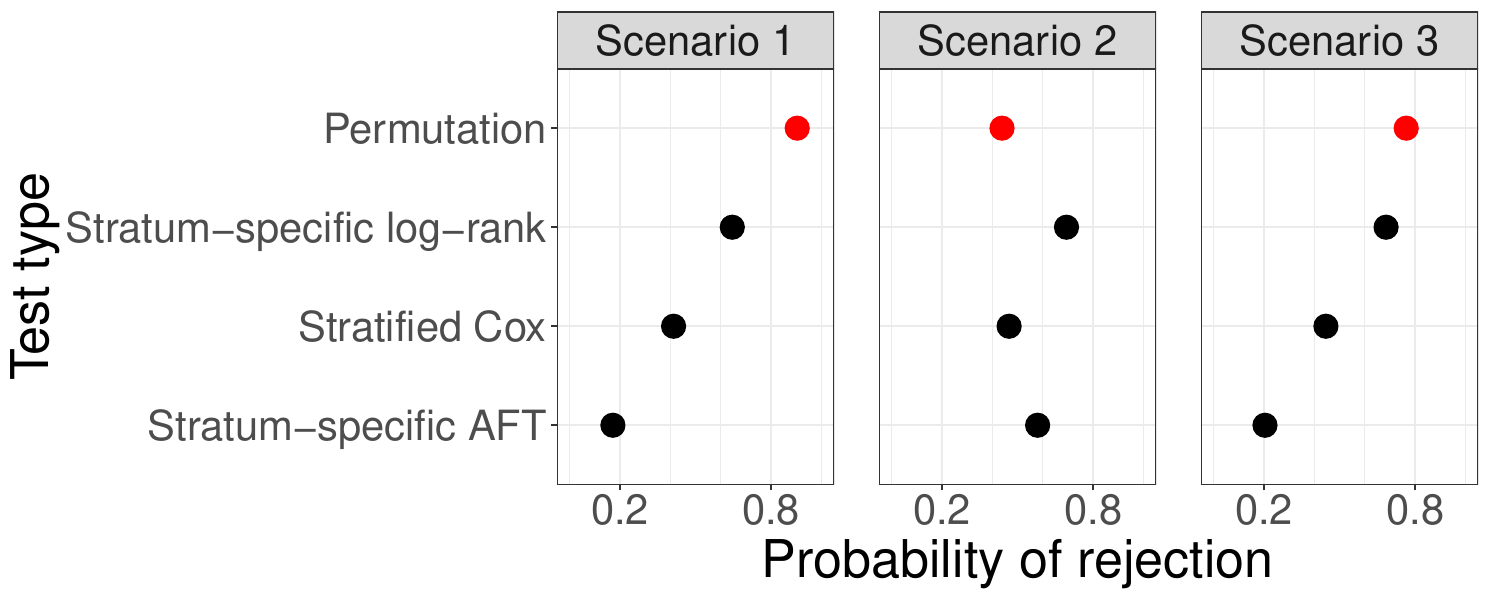}
\end{figure}

\end{document}